\documentclass[10pt,oneside,a4paper]{article}

\usepackage{amssymb}
\usepackage{amsmath}
\usepackage{amsthm}
\usepackage{color}
\usepackage{graphicx}
\usepackage{wrapfig}
\usepackage{float}
\usepackage{subfigure}
\usepackage{rotating}
\usepackage{cite}
\usepackage{caption}
\usepackage{hyperref}
\usepackage{enumitem}
\usepackage{varwidth}

\newtheorem{thm}{Theorem}[section]
\newtheorem{cor}[thm]{Corollary}

\newtheorem{lem}[thm]{Lemma}

\theoremstyle{remark}
\newtheorem{rem}[thm]{Remark}

\theoremstyle{definition}

\def\RR{\mathbb{R}}

\def\ZZ{\mathbb{Z}}
\def\TT{\mathbb{T}}
\def\CritSet{\mathcal{C}}

\DeclareMathOperator{\arccot}{arccot}

\newcommand{\comment}[1]{}
\let\oldmarginpar\marginpar
\renewcommand\marginpar[1]{\-\oldmarginpar[\raggedleft\footnotesize #1]
{\raggedright\footnotesize #1}}

\author{
Jordi-Llu\'is Figueras
\thanks
{
Department of Mathematics, Uppsala University, 
Box 480, 75106 Uppsala (Sweden). {\tt figueras@math.uu.se}.
}
\and
Thomas Ohlson Timoudas
\thanks
{
Department of Mathematics, KTH Royal Institute of Technology, 
Institutionen f\"or matematik, KTH, 100 44 Stockholm (Sweden). {\tt ttohlson@kth.se}.
}
}

\title{Sharp $\frac12$-H\"older continuity of the Lyapunov exponent at
the bottom of the spectrum for a class of Schr\"odinger cocycles}

\date{}

\begin{document}
	
\maketitle

\begin{abstract}
We consider a similar type of scenario for the disappearance of uniform of 
hyperbolicity as in Bjerkl\"ov and Saprykina (2008 \emph{Nonlinearity} \textbf{21}), 
where it was proved that the minimum distance between invariant stable 
and unstable bundles has a linear power law dependence on parameters. In 
this scenario we prove that the Lyapunov exponent is sharp $\frac12$-H\"older continuous.

In particular, we show that the Lyapunov exponent of Schr\"odinger cocycles with a potential having a unique non-degenerate minimum, is sharp $\frac12$-H\"older continuous below the lowest energy of the spectrum, in the large coupling regime.
\end{abstract}

\section{Introduction.}\label{section: introduction}

We consider quasi-periodic cocycles 
\begin{equation}\label{eq:linear cocycle}
\left\{
\begin{array}{lcl}
\mathbb{T}\times\mathbb R^2&\longrightarrow& \mathbb{T}\times\mathbb R^2\\
(\theta, v)&\longrightarrow& (\theta+\omega, A(\theta)v)
\end{array}
\right.
\end{equation}
where $\omega$ is an irrational number and $A\in\mathcal{C}^1(\mathbb T, \text{SL}(2, \mathbb R))$. A cocycle is uniformly hyperbolic if there exists two continuous maps $W_i:\mathbb T\rightarrow \mathbb R^2,$ $i=1,2$, that are invariant, $A(\theta) W_i(\theta)=\lambda_i(\theta) W_i(\theta+\omega)$, with $\lambda_i:\mathbb T\rightarrow \mathbb R-\left\{0\right\}$ and
\begin{equation}
\left|\int_{\mathbb T}\log|\lambda_i(\theta)|d\theta\right| > 0.\label{eq: UHCondition}
\end{equation}
Since we consider only $\text{SL}(2, \mathbb R)$ cocycles, one of these integrals will be positive, and the other integral negative (simply minus the first one). Without loss of generality, we may assume that
\begin{equation}
\int_{\mathbb T}\log|\lambda_1(\theta)|d\theta > 0,\label{eq: PositiveExponent}
\end{equation}
and we call the quantity $\int_{\mathbb T}\log|\lambda_1(\theta)|d\theta$, the (maximal) Lyapunov exponent.

In this paper we tackle the problem of H\"older continuity of the Lyapunov exponent of one-parametric families of quasi-periodic cocyles with transfer matrix of the form 
\[
A_E(\theta)=
\begin{pmatrix}
V(\theta)-E & -1\\
1 & 0
\end{pmatrix},  
\]
where the function $V \in \mathcal{C}^2(\TT, \RR)$ has a non-degenerate minimum. Since these cocycles arise from the study of the spectrum of discrete Schr\"odinger operators $H_\theta:\ell^2(\mathbb Z)\rightarrow \ell^2(\mathbb Z)$ of the form 
\begin{equation}\label{eq: sch operator}
(H_\theta x)_n = x_{n+1}+V(\theta+n\omega) x_n+x_{n-1},
\end{equation}
it is known that the set of parameters $E$ such that the cocycle is not uniformly hyperbolic is a compact set on the real line. From now on, we will denote by $E_0$ the bottom edge of this set.

More concretely, in this paper we prove that on the whole interval $(-\infty, E_0]$ the Lyapunov exponent is $\frac12$-H\"older continuous but not $(\frac12+\delta)$-H\"older continuous for any positive $\delta$. Since the Lyapunov exponent is real analytic in the interval $(-\infty, E_0)$ it implies that its asymptotic behaviour at $E_0$ is square root like. This problem is part of the open conjectures appearing in \cite{Haro_Llave_06} on the asymptotics of disappearance of normally hyperbolic invariant tori in quasi-periodically forced systems (see also\cite{Calleja_Figueras_12, Figueras_Haro_15, Figueras_Haro_16} for further numerical studies in different contexts). There, the authors numerically study one-parameter families of quasi-periodic skew product systems, and conjectured that the normal dynamics around the invariant tori satisfy that the minimum distance between the invariant bundles $W_1$ and $W_2$ (defined above) satisfies
\begin{equation}
\text{dist}(E)\sim a(E-E_0),\label{eq: LinearAsymptoticsConjecture} 
\end{equation}
and that the (maximal) Lyapunov exponent $L(E)$ of $A_E$, satisfies
\begin{equation}
L(E)\sim L(E_0)+b\sqrt{E_0-E},\label{eq: LEAsymptoticsConjecture}
\end{equation}
where $a, b\neq 0$ are constants depending on the system and $L(E_0) > 0$.

A step forward towards solving this problem was taken in \cite{Bjerklov_Saprikina_08}, where the authors proved that the distance is asymptotically linear in the case of linear H\'enon maps. These are 
one-parameter linear cocycles of the form \eqref{eq:linear cocycle} with transfer matrix
\[
A_E(\theta)=
\begin{pmatrix}
V(\theta)-E & 1\\
1 & 0
\end{pmatrix},
\]
with $V$ being any $\mathcal C^2$ function close to $V(\theta)=\frac{1}{1+\lambda^2\sin^2(\pi x)}$, $\lambda\gg 1$ and $E$ being a parameter.

\begin{rem}
A similar result as in \cite{Bjerklov_Saprikina_08} with the same type 
of potential $V$ as in this paper appears in \cite{OhlsonTimoudas_18}.
\end{rem}

In the literature there are plenty of results about the H\"older continuity of the Lyapunov exponent. Among them, in \cite{Jitormiskaya_Koslover_Schulteis_09} the authors prove that, if the cocycle is analytic, then the Lyapunov exponent is continuous with respect the parameter $E$. In \cite{Goldstein_Schlag_2001}
they established that it is H\"older continuous.

Proving asymptotics for the distance and Lyapunov exponents when the potential 
$V$ is close to zero
is quite easy: both are square root. This can be proved by
noticing that the collision of the invariant stable and unstable bundles is
smooth, so producing a saddle-node bifurcation, which implies that both the 
distance between the invariant bundles and the Lyapunov exponent have 
square root asymptotics.

We prove, under the assumption that the minimum distance satisfies linear asymptotics, as in \eqref{eq: LinearAsymptoticsConjecture}, together with some other general assumptions, that the Lyapunov exponent has (almost) square-root asymptotic behaviour, as in \eqref{eq: LEAsymptoticsConjecture}, but where the constant is allowed to fluctuate between two fixed, positive constants.

\section{Statement of the results}

Suppose that $(\omega, A_t(\theta))$ is an analytic family of quasi-periodic cocycles $A_t \in \mathcal C^1(\RR \times \TT, \text{SL}(2, \RR))$, given by
\begin{equation}
A_t(\theta) = A(\theta) e^{tw},\label{eq: ParameterDependence}
\end{equation}
where $t$ is a real parameter, $A \in \mathcal C^1(\TT, \text{SL}(2, \RR))$,  and $w = \begin{pmatrix} 0 & 0 \\ 1 & 0 \end{pmatrix}$. The results in this paper hold for any such family of cocycles that satisfies the following list of assumptions. The (maximal) Lyapunov exponent of $A_t$ is the limit
\begin{equation*}
L(t) = \lim \limits_{n \to \infty} \frac1n \int \limits_\TT \log \| A^n_t(\theta) \| d\theta
\end{equation*}
where $A^n(\theta) = A(\theta + (n-1)\omega) \cdots A(\theta)$. If there are two continuous maps $W_i: \TT \to \RR^2$ ($i = 1,2$), that are invariant, $A_t(\theta) W_i(\theta)=\lambda_i(\theta) W_i(\theta+\omega)$, as in \eqref{eq: UHCondition}, then $L(t)$ equals the biggest of the expressions
\begin{equation*}
\lim \limits_{n \to \infty} \frac1n \sum \limits_{j = 0}^{n - 1} \int \limits_\TT \log|\lambda_i(\theta + j \omega)| d\theta.
\end{equation*}
As in the introduction, we may assume that the biggest of them has $i = 1$. Then, it follows that
\begin{equation*}
L(t) = \int_{\mathbb T}\log|\lambda_1(\theta)|d\theta.
\end{equation*}

\subsection{Assumptions}

The first assumption we make is that the cocycle is uniformly hyperbolic, up to some (critical) parameter, and that there is an appropriate invariant cone. From the first assumption, it is not a priori clear that $t_0$ is a bifurcation parameter, but it will follow from the second assumption. We remark that invariant cone condition could be relaxed somewhat, but would require further technical arguments. Here, we use projective coordinates, but other coordinates could be used instead.

\begin{enumerate}[label={A\arabic*$(\epsilon, t_0)$}, ref=A\arabic*]
\item\label{AssInvCone} Suppose that the cocycles $(\omega, A_t)$ are uniformly hyperbolic for every $t \in (t_0 - \epsilon, t_0)$. Moreover, suppose that there are two continuous functions $r^u_t, r^s_t: \TT \to [\frac1C, C]$, where $C > 0$ is independent of $t$, such that
\begin{align*}
\begin{pmatrix} r^u_t(\theta) \\ 1 \end{pmatrix} \text{ and } \begin{pmatrix} r^s_t(\theta) \\ 1 \end{pmatrix}
\end{align*}
are invariant directions for $A_t$, and that they lie in the subspaces $W_1(\theta)$ and $W_2(\theta)$, respectively.
\end{enumerate}
We write the difference between the two functions $r^u_t$ and $r^s_t$ at the point $\theta$ as
\begin{align*}
d(\theta) = r^u_t(\theta) - r^s_t(\theta).
\end{align*}
Because of the fibred structure of the cocycle, the minimum of $d(\theta)$ is the minimum distance between the invariant bundles (in projective coordinates). Whenever $i \leq j$, there is a unique $D_{i,j}(\theta)$ such that
\begin{align*}
d(\theta + j\omega) = d(\theta + i\omega) \cdot D_{i, j}(\theta).
\end{align*}
That is, $D_{i, j}(\theta)$ measures how the difference changes from step $i$ to $j$, where the base step 0 is taken to be at $\theta$. We will make some assumptions about the function $d(\theta)$, and some global estimates for the growth of $D_{i, j}(\theta)$. Let $\epsilon$ be the constant in \ref{AssInvCone}.

\begin{enumerate}[resume, label={A\arabic*$(\epsilon, t_0)$}, ref=A\arabic*]
\item\label{AssInterval} Suppose that, for every $t \in (t_0 - \epsilon, t_0)$, there is a distinguished interval $I = I(t) \subset \TT$, and a distinguished point $\theta_c = \theta_c(t) \in I$, satisfying the following conditions. 
\begin{enumerate}
\item\label{AssExponentialGrowth} Suppose that we are given stopping times $\sigma^\pm = \sigma^\pm(t, \theta) > 0$ such that, for every $\theta \in I$ and $0 < j \leq \sigma^+, 0 < k \leq \sigma^-$, we have
\begin{align}
D_{0, j-1}(\theta) &\geq e^{aj}, \text{ and }\label{eq: AssForwardGrowthBound}\\
D_{-(k-1), 0}(\theta) &\leq e^{-ak},\label{eq: AssBackwardGrowthBound}
\end{align}
where $a > 0$ is independent of both $t$ and $\theta$.
\item\label{AssLinearApproach} There is a unique (global) minimum distance, it is attained at $\theta_c$, and satisfies the linear asymptotics
\begin{align}
d(\theta_c) = \min \limits_{\theta \in \TT} d(\theta) = C_0 |t - t_0| + o(|t - t_0|),\label{eq: LinearDistance}
\end{align}
for some constant $C_0 > 0$ independent of $t$.
\item\label{AssApproxQuadratic} There is a constant $C_1 > 0$, independent of both $t$ and $\theta$, such that for every $\theta \in I$ the distance is approximately quadratic:
\begin{align}
\frac1{C_1} (\theta - \theta_c)^2 \leq d(\theta) - d(\theta_c) \leq C_1 (\theta - \theta_c)^2.\label{eq: AssQuadratic}
\end{align}
\item\label{AssIntervalLength} The length of the interval satisfies
\begin{align}
|I| \geq 2C_2 \cdot \sqrt{d(\theta_c)},\label{eq: AssIntervalLength}
\end{align}
where $C_2 > 0$ is some constant independent of $t$.
\item\label{AssDIstanceOutsideCritical} For every $\theta \in \TT \backslash \{ \theta + k\omega : \theta \in I, \sigma^-(\theta) \leq k \leq \sigma^+(\theta)\}$, the difference satisfies
\begin{align}
d(\theta) \geq \sqrt{d(\theta_c)}.\label{AssOutsideCritical}
\end{align}
\end{enumerate}
\end{enumerate}

It is important to stress that \eqref{eq: LinearDistance} implies that $\min \limits_{\theta \in \TT} d(\theta) \to 0$ as $t \nearrow t_0$, and that $t_0$ is indeed a bifurcation point where uniform hyperbolicity is lost. That is, we have a torus collision at the critical parameter $t_0$.

Lastly, we impose a continuity condition on the Lyapunov exponent.
\begin{enumerate}[resume, label={A\arabic*$(\epsilon, t_0)$}, ref=A\arabic*]
\item\label{AssLEContinuity} Suppose that $L(t)$ is continuous on the interval $(t_0 - \epsilon, t_0]$. That is,
\begin{align}
\lim \limits_{t \nearrow t_0} L(t) = L(t_0).
\end{align}
\end{enumerate}
Since the parameter dependence is analytic, the Lyapunov exponent is continuous at parameters $t$ satisfying that $A_{t}$ is uniformly hyperbolic. Therefore, the assumption is in fact only that $L(t)$ is left-continuous at $t_0$.

We will discuss these assumptions in Section \ref{section: discussion}.

\subsection{Main results}

In this paper we prove

\begin{thm}\label{thm:1}
Given a one-parameter family of cocycles $A_t$ of the form in \ref{eq: ParameterDependence}, and satisfying the assumptions \ref{AssInvCone} and \ref{AssInterval}, for some $\epsilon > 0$ and $t_0$, there exist two positive constants $K_1$ and $K_2$, such that
\begin{equation}\label{eq: bounds}
\frac{K_1}{\sqrt{|t - t_0|}}\leq \frac{d}{dt} L(t)\leq 
\frac{K_2}{\sqrt{|t - t_0|}},
\end{equation}
for every $t \in (t_0 - \epsilon, t_0)$.
\end{thm}
A direct result of this is the following.
\begin{cor}
Under the same assumptions as in Theorem \ref{thm:1}, together with the assumption \ref{AssLEContinuity}, the Lyapunov exponent in the range $(t_0 - \epsilon, t_0]$ is $\frac12$-H\"older continuous but not 
$(\frac12+\delta)$-H\"older continuous at $t_0$ for any $\delta > 0$. More specifically, we have the asymptotics
\begin{align*}
K_1 \sqrt{|t - t_0|} \leq L(t) - L(t_0) \leq K_2 \sqrt{|t - t_0|},
\end{align*}
for every $t \in (t_0 - \epsilon, t_0)$.
\end{cor}

\begin{proof}
Since the system is uniformly hyperbolic in $(t_0 - \epsilon, t_0)$, the Lyapunov exponent is analytic there. The only point of interest is $t_0$. Let $t \in (t_0 - \epsilon, t_0)$. Since the Lyapunov exponent is left-continuous at $t_0$, by assumption \ref{AssLEContinuity}, we may simply integrate the derivative from $t$ to $t_0$ to get the result.
\end{proof}

In the setting of Schr\"odinger cocycles, we have the following corollary. Recall that an irrational $\omega$ is called Diophantine if there are constants $\kappa > 0$ and $\tau \geq 1$ such that
\[
\inf \limits_{p \in \ZZ} |n \omega - p| \geq \frac \kappa{|n|^\tau},
\]
for every $n \in \ZZ - \{0\}$.
\begin{cor}
Suppose that $(\omega, A_E)$ is a family of Schr\"odinger cocycles where $\omega$ is a Diophantine irrational,
\begin{align*}
A_E = \begin{pmatrix} \lambda V(\theta) - E & 1 \\ 1 & 0 \end{pmatrix},
\end{align*}
and $V(\theta) \in \mathcal{C}^2(\TT, \RR)$ has a unique, non-degenerate minimum. Then there is a $\lambda_0$ such that if $\lambda \geq \lambda_0$, the Lyapunov exponent satisfies the asymptotics
\begin{align*}
K_1 \sqrt{|E - E_0|} \leq L(E) - L(E_0) \leq K_2 \sqrt{|E - E_0|},
\end{align*}
as $E \nearrow E_0$, where $K_1 \leq K_2$ are some positive constants and $E_0$ is the lowest energy of the spectrum.

\end{cor}

\begin{proof}
In this setting, the Lyapunov exponent is continuous up to the bottom edge of the spectrum, that is, \ref{AssLEContinuity} holds. This fact can be found in \cite{Johnson_1984}, where it is proved in the continuous case, but it applies also to the discrete one.

The assumptions \ref{AssInvCone} and \ref{AssInterval} are proved in \cite{OhlsonTimoudas_18}, for $t_0 = E_0$, and $\epsilon = 1$. Specifically, there is an interval $I(E)$ for each energy $E < E_0$, where $E_0$ is the lowest energy of the spectrum. The functions $r^u_E$ and $r^s_E$ are the functions $\psi^u_E$ and $\psi^s_E$ in that paper, respectively. In that paper, ($\mathcal{A}$1) corresponds to \ref{AssInvCone}, and ($\mathcal{A}$2) corresponds to the assumptions \ref{AssLinearApproach} to \ref{AssIntervalLength}. Moreover, the first estimate in ($\mathcal{A}$3) (b) is the same one as in \eqref{eq: AssForwardGrowthBound} for $\theta \in I(E) = I$. Similarly, the first estimate in ($\mathcal{A}$3) (c) is the one in \eqref{eq: AssBackwardGrowthBound}. That shows \ref{AssExponentialGrowth}. The last assumption \ref{AssDIstanceOutsideCritical} corresponds to \cite[Lemma 7.6]{OhlsonTimoudas_18}.
\end{proof}

The proof of Theorem \ref{thm:1} appears in Section \ref{sec: ProofMainTheorem}, but first we prove a small lemma.

\section{The derivative of the Lyapunov exponent and Avila's lemma}

One of the key tools for proving Theorem \ref{thm:1} is expressing the derivative 
of the Lyapunov exponent as an integral with respect to the stable and unstable 
directions, and the difference between them. Since, for parameters below $t_0$, the 
cocycle is uniformly hyperbolic, it means that there exists a $B:\mathbb T\rightarrow \text{SL}(2, \mathbb R)$, with the same regularity 
as the cocycle, 
\[
B_t(\theta)= \begin{pmatrix}
\alpha_t(\theta) & \beta_t(\theta)\\
\gamma_t(\theta) & \delta_t(\theta)
\end{pmatrix}
\]
such that $B(\theta+\omega)^{-1} A_t(\theta) B(\theta)$ is 
the diagonal transfer matrix
\[
D_t(\theta)=\begin{pmatrix}
d_{1,t}(\theta) & 0\\
0 & d_{2,t}(\theta)
\end{pmatrix}, 
\] 
with $d_1(\theta) > d_2(\theta)$.
In fact, since the dynamics on the base is irrational, the Lyapunov exponent is given by 
\[
L(t) = \int_{\mathbb T} \log(d_{1,t}(\theta)) d\theta.
\]

\begin{lem}\label{lemma:1}
Given a one-parameter family of uniformly hyperbolic quasi-periodic cocycles $(\omega, A_t)$, of 
the form \eqref{eq:linear cocycle}, with $A_t\in \mathcal C^1(\mathbb R\times 
\mathbb T, \text{SL}(2, \mathbb R))$, and $A_t(\theta)=A_0(\theta)e^{t w(\theta)}$, 
where $w(\theta)\in sl(2, \mathbb R)$ is given by
\[
w(\theta)=\begin{pmatrix}
w_1(\theta) & w_2(\theta)\\
w_3(\theta) & -w_1(\theta)
\end{pmatrix},
\]
then 
\[
\frac{d}{dt}{L}_{|t=0} = \sum_{i=1}^3\int_{\mathbb T}q_i(\theta) w_i(\theta) d\theta, 
\]
where 
\[
q_1(\theta)=\alpha_0(\theta)\delta_0(\theta)+\beta_0(\theta)\gamma_0(\theta), 
q_2(\theta)=\gamma_0(\theta)\delta_0(\theta), \text{ and }
q_3(\theta)=-\beta_0(\theta)\alpha_0(\theta).
\]
\end{lem}
This lemma appears in \cite{Avila_15} in the case that the transfer matrix 
is analytic. For completeness sake, we include a slightly different proof of this lemma.

\begin{rem}
The proof of Lemma \ref{lemma:1} can be generalized mutatis mutandis for one-parameter
families of uniformly hyperbolic cocycles acting on $M\times \mathbb R^n$, 
where $M$ is a compact manifold and $n\geq 2$, with base dynamics on 
$M$ satisfying that its jacobian at any point has determinant 1.
\end{rem}

\begin{proof}
Under the assumptions of the lemma, 
\begin{equation}
\label{eq: conjugation2}
B(\theta+\omega)^{-1}A(\theta)e^{t w(\theta)}B(\theta)=D(\theta), 
\end{equation}
where the matrix-valued maps $B$ and $D$ also depend on the parameter 
$t$. 

First notice that 
\[
\frac{d}{dt}{L}_{|t=0} = \int_{\mathbb T} \frac{d'_1(\theta)}{d_1(\theta)}d\theta. 
\]
So, by differentiating \eqref{eq: conjugation2} with respect $t$ and setting
$t=0$ we obtain 
\begin{equation}
\label{eq: 1}
(B(\theta+\omega)^{-1})'A(\theta)B(\theta)+
B(\theta+\omega)^{-1}A(\theta)w(\theta)B(\theta)+
B(\theta+\omega)^{-1}A(\theta)B'(\theta)
=D'(\theta).
\end{equation}
Also, by differentiating $B(\theta)^{-1} B(\theta)=\text{Id}$ with respect $t$ we get
$(B(\theta)^{-1})'=-B(\theta)^{-1}B'(\theta)B(\theta)^{-1}$, and using identity 
\eqref{eq: conjugation2},  Equation \eqref{eq: 1} is
transformed into
\begin{equation}
-B(\theta+\omega)^{-1}B'(\theta+\omega)D(\theta)+
D(\theta)B(\theta)^{-1}w(\theta)B(\theta)+
D(\theta)B(\theta)^{-1}B'(\theta)
=D'(\theta).
\end{equation}
Finally, by dividing both sides by $d_1(\theta)$ in the last equation, 
considering the average of the $(1,1)$ entry, and 
using the fact that the first and third monomials of the left-hand side
cancel out, we get the desired result.
\end{proof}

\section{Proof of Main Theorem}\label{sec: ProofMainTheorem}

\begin{proof}[Proof of Theorem \ref{thm:1}]
Suppose that all the assumptions are satisfied, and fix a $t \in (t_0 - \epsilon, t_0)$. Let us drop $t$ from the notation and simply write $r_u$ and $r_s$ instead of $r^u_t$ and $r^s_t$, respectively. In our setting, 
\[
B_t(\theta)=
\begin{pmatrix}
\frac{r_u(\theta)}{\sqrt{r_u(\theta)-r_s(\theta)}} & 
\frac{r_s(\theta)}{\sqrt{r_u(\theta)-r_s(\theta)}} \\ 
\frac{1}{\sqrt{r_u(\theta)-r_s(\theta)}} & 
\frac{1}{\sqrt{r_u(\theta)-r_s(\theta)}} \\ 
\end{pmatrix},
\text{ and }
w(\theta)=\begin{pmatrix}0& 0\\ 1 & 0 \end{pmatrix}.
\]
Lemma \ref{lemma:1} can be used to express the derivative of $L(t)$ at any $t$, simply by shifting the parameter, since the formula doesn't depend on $A_0$. That is, we can simply apply the lemma to $A_{t+s} = (A_0e^{tw})e^{sw}$, and get the derivative at the parameter value $t$. The derivative of the Lyapunov exponent is therefore equal to  
\[
-\int_\mathbb T \dfrac{r_u(\theta)r_s(\theta)}{r_u(\theta)-r_s(\theta)}d\theta.
\]
Recall the notation
\begin{align*}
d(\theta) = r_u(\theta) - r_s(\theta)
\end{align*}
and $D_{i,j}(\theta)$, where $i \leq j$, that satisfy the relations
\begin{align*}
d(\theta + j\omega) = d(\theta + i\omega) \cdot D_{i, j}(\theta).
\end{align*}
First of all, for every $t \in(t_0 - \epsilon, t_0)$, we have the uniform inequalities
\begin{align*}
\frac1{2C} \int \limits_{\TT} \frac{1}{d(\theta)}d\theta \leq \int \limits_{\TT} \frac{r_u(\theta)r_s(\theta)}{r_u(\theta) - r_s(\theta)}d\theta \leq 2C \int \limits_{\TT} \frac{1}{d(\theta)}d\theta,
\end{align*}
since there is a constant $C > 0$, independent of $t$ and $\theta$, such that $r_s, r_u \in [\frac1C, C]$, for every $\theta$. In particular, this means that the result follows, if we can show that the same inequalities hold for the integral
\begin{align*}
\int \limits_{\TT} \frac{1}{d(\theta)}d\theta.
\end{align*}
Recall the interval $I = I(E) \subseteq \TT$, and consider the transformation
\begin{align*}
\int \limits_{I + k\omega} \frac{1}{d(\theta)}d\theta = \int \limits_{I} \frac{1}{d(\theta + k\omega)}d\theta.
\end{align*}
Thus the integral over $I + k\omega$, where $k > 0$, becomes
\begin{align*}
\int \limits_{I} \frac{1}{d(\theta)D_{0, k}(\theta)}d\theta,
\end{align*}
and over $I - k\omega$, where again $k > 0$, it becomes
\begin{align*}
\int \limits_{I} \frac{D_{-k, 0}(\theta)}{d(\theta)}d\theta.
\end{align*}
Recall the stopping times $\sigma^\pm(\theta) > 0$ from the assumptions, where $\theta \in I$. Set
\begin{align*}
\CritSet^+ &= \{ \theta + k\omega : \theta \in I, 0 < k \leq \sigma^+(\theta) \},\\
\CritSet^- &= \{ \theta - k\omega : \theta \in I, 0 < k \leq \sigma^-(\theta) \}, \text{ and}\\
\CritSet &= I \cup \CritSet^+ \cup \CritSet^-.
\end{align*}
Using the relations above, we conclude that
\begin{align*}
\int \limits_\CritSet \frac1{d(\theta)} d\theta
= \int \limits_I \frac1{d(\theta)} \Bigg(1 + \sum \limits_{k = 1}^{\sigma^+(\theta)} \frac1{D_{0, k}(\theta)} + \sum \limits_{k = 1}^{\sigma^-(\theta)} D_{-k, 0}(\theta) \Bigg)
d\theta.
\end{align*}
The assumption \ref{AssExponentialGrowth} yields the inequalities
\begin{align*}
\int \limits_I \frac1{d(\theta)} \leq \int \limits_\CritSet \frac1{d(\theta)} d\theta
\leq \bigg(1 + \frac2{1 - e^{-a}} \bigg) \int \limits_I \frac1{d(\theta)} d\theta.
\end{align*}
Indeed, if we set
\begin{align*}
S^+(\theta) = \sum \limits_{k = 1}^{\sigma^+(\theta)} \frac1{D_{0, k-1}(\theta)}, \text{ and } S^-(\theta) = \sum \limits_{k = 1}^{\sigma^-(\theta)} D_{-(k - 1), 0}(\theta),
\end{align*}
then the bounds in \eqref{eq: AssForwardGrowthBound} and \eqref{eq: AssBackwardGrowthBound} give us
\begin{align*}
0 \leq S^\pm(\theta) \leq \sum \limits_{k=1}^{\infty} e^{-ak} = \frac1{1 - e^{-a}},
\end{align*}
which immediately imply the above bounds. The assumption \ref{AssApproxQuadratic} implies that
\begin{align}
\int \limits_I \frac1{d(\theta_c) + \frac1{C_1}(\theta - \theta_c)^2} d\theta \leq \int \limits_I \frac1{d(\theta)} d\theta \leq \int \limits_I \frac1{d(\theta_c) + C_1(\theta - \theta_c)^2} d\theta.\label{eq: integralQuadraticExpansion}
\end{align}
For any $b > 0$, we compute the indefinite integral
\begin{align*}
\int \limits \frac1{d(\theta_c) + b \cdot (\theta - \theta_c)^2} d\theta = \frac1{\sqrt{b \cdot d(\theta_c)}} \arctan\Bigg( (\theta - \theta_c) \sqrt{\frac{b}{d(\theta_c)}} \Bigg).
\end{align*}
By assumption \ref{AssIntervalLength}, it follows that either $I \supseteq [d(\theta_c), d(\theta_c) + C_2\sqrt{d(\theta_c)}]$ or $I \supseteq [d(\theta_c) - C_2\sqrt{d(\theta_c)}, d(\theta_c)]$. Since the problem is symmetric, it doesn't matter which inclusion holds. Computing the integral over the interval $I$, we obtain the inequality
\begin{align}
\frac1{\sqrt{b \cdot d(\theta_c)}} \arctan( \sqrt{b} ) \leq \int \limits_I \frac1{d(\theta_c) + b \cdot (\theta - \theta_c)^2} d\theta \leq \frac\pi{\sqrt{d(\theta_c)}}.\label{IntCritInverseSquareRootBounds}
\end{align}
Plugging in the constant $\frac1{C_1}$ instead of $b$, and using the inequality in (\ref{eq: integralQuadraticExpansion}), results in the bounds
\begin{align*}
\frac{\sqrt{C_1}}{\sqrt{d(\theta_c)}} \arctan( \sqrt{\frac1{C_1}} ) \leq \int \limits_\CritSet \frac1{d(\theta)} d\theta
\leq \bigg(1 + \frac2{1 - e^{-a}} \bigg) \frac\pi{\sqrt{d(\theta_c)}}.
\end{align*}
Finally, using the assumption \ref{AssDIstanceOutsideCritical}, the remainder of the integral can be computed as
\begin{align*}
0 \leq \int \limits_{\TT \backslash \CritSet} \frac1{d(\theta)} \leq \frac{1}{\sqrt{d(\theta_c)}}.
\end{align*}
In conclusion, there is a constant $K > 0$ (uniform in $t$) such that
\begin{align*}
\frac1{K\sqrt{d(\theta_c)}} \leq \int \limits_\TT \frac1{d(\theta)} d\theta \leq \frac{K}{\sqrt{d(\theta_c)}}.
\end{align*}
Since $d(\theta_c) = C_0 \cdot |t - t_0| + o(|t - t_0|)$, by assumption \eqref{AssLinearApproach}, it follows that
\begin{align*}
\frac{K_1}{\sqrt{|t - t_0|}}\leq \frac{d}{dt} L(t) \leq \frac{K_2}{\sqrt{|t - t_0|}},
\end{align*}
where $K_2 \geq K_1 > 0$ are independent of $t$, provided $t$ is close enough to $t_0$.
\end{proof}

\section{Discussion and future directions}\label{section: discussion}

The first assumption \ref{AssInvCone} holds for Schr\"odinger cocycles that are homotopic to the identity, for energies below and above the lowest and highest energies of the spectrum, respectively. For a proof, see \cite{Herman_1983}. The important implication of this is that there is no fibre rotation, and the assumption should apply to many classes of cocycles.

It should be possible to extend the methods to include more general cases, where fibre rotation is allowed.

It is clear that the bounds
\begin{align*}
\frac1{C\sqrt{d(\theta_c)}} \leq \int \limits_I \frac1{d(\theta)} d\theta \leq \frac{C}{\sqrt{d(\theta_c)}},
\end{align*}
critically depend on the assumptions \ref{AssApproxQuadratic} and \ref{AssIntervalLength}.
(see the computations leading up to \eqref{IntCritInverseSquareRootBounds}). In some computed examples (see \cite{OhlsonTimoudas_17, OhlsonTimoudas_18}), the length of the interval actually satisfies much better bounds ($C_2$ can be chosen arbitrarily large as $t \nearrow t_0$). Different local behaviour of their difference, that is replacing the square in \ref{AssApproxQuadratic} by another exponent, should produce different types of asymptotics.

The linear asymptotics  for the distance in assumption \ref{AssLinearApproach} enters in the last step, and also affect the final asymptotics. As we said in \eqref{eq: LinearAsymptoticsConjecture}, the minimum angle between the invariant directions is conjectured to be linear when a system loses uniform hyperbolicity (see for instance \cite{Haro_Llave_06}).

\begin{rem}
The assumptions \ref{AssInvCone} and \ref{AssLinearApproach} together imply that the angles between the invariant directions satisfies linear asymptotics. Indeed, this angle is simply the difference between $\arccot(r^s_E)$ and $\arccot(r^u_E)$, and the mean-value theorem gives the linear asymptotics since $r^s_E, r^u_E \in [\frac1C, C]$.
\end{rem}

The assumption \ref{AssExponentialGrowth} is simply a statement of exponential growth when the invariant directions are very close to each other. It can be relaxed somewhat, as long as their sums
\begin{align*}
S^+(\theta) = \sum \limits_{k = 1}^{\sigma^+(\theta)} \frac1{D_{0, k-1}(\theta)}, \text{ and } S^-(\theta) = \sum \limits_{k = 1}^{\sigma^-(\theta)} D_{-(k - 1), 0}(\theta).
\end{align*}
are uniformly bounded.

The assumption \ref{AssDIstanceOutsideCritical} simply ensures that directions are not too close to each other outside some critical region. In known examples satisfying these assumptions, \eqref{AssOutsideCritical} can in fact be replaced with $d(\theta) \gg d(\theta_c)$.

As for the last assumption \ref{AssLEContinuity}, we refer to \cite{Johnson_1984} for the case of Schr\"odinger cocycles. There it is proved to be continuous in the open spectral gaps, for the continuous case, but it applies also to the discrete case. The Lyapunov exponent for Schr\"odinger cocycles is continuous in $E$ for analytic potentials (see \cite{Jitormiskaya_Koslover_Schulteis_09}), but may fail to be so for non-analytic potentials (see for instance \cite{Wang_You_13}).

\section*{Acknowledgments.}
The authors acknowledge fruitful discussions with Michael Benedicks, Kristian Bjerkl\"ov, Alex Haro, Rafael 
de la Llave, and Joaquim Puig. 
\bibliography{./bibliography}{} 

\begin{thebibliography}{10}

\bibitem{Avila_15}
A.~Avila.
\newblock Global theory of one-frequency {S}chr\"odinger operators.
\newblock {\em Acta Math.}, 215(1):1--54, 2015.

\bibitem{Bjerklov_Saprikina_08}
K.~Bjerkl{\"o}v and M.~Saprykina.
\newblock Universal asymptotics in hyperbolicity breakdown.
\newblock {\em Nonlinearity}, 21(3):557--586, 2008.

\bibitem{Calleja_Figueras_12}
R.~Calleja and J.-L. Figueras.
\newblock Collision of invariant bundles of quasi-periodic attractors in the
  dissipative standard map.
\newblock {\em Chaos}, 22(3):033114, 10, 2012.

\bibitem{Figueras_Haro_15}
J.-L. Figueras and A.~Haro.
\newblock Different scenarios for hyperbolicity breakdown in quasiperiodic area
  preserving twist maps.
\newblock {\em Chaos}, 25(12):123119, 16, 2015.

\bibitem{Figueras_Haro_16}
J.-L. Figueras and A.~Haro.
\newblock A note on the fractalization of saddle invariant curves in
  quasiperiodic systems.
\newblock {\em Discrete Contin. Dyn. Syst. Ser. S}, 9(4):1095--1107, 2016.

\bibitem{Goldstein_Schlag_2001}
M.~Goldstein and W.~Schlag.
\newblock H\"older continuity of the integrated density of states for
  quasi-periodic {S}chr\"odinger equations and averages of shifts of
  subharmonic functions.
\newblock {\em Ann. of Math. (2)}, 154(1):155--203, 2001.

\bibitem{Haro_Llave_06}
A.~Haro and R.~de~la Llave.
\newblock Manifolds on the verge of a hyperbolicity breakdown.
\newblock {\em Chaos}, 16(1):013120, 8, 2006.

\bibitem{Herman_1983}
M.-R. Herman.
\newblock Une m\'ethode pour minorer les exposants de lyapounov et quelques
  exemples montrant le caract\'ere local d'un th\'eor\`eme d'arnold et de moser
  sur le tore de dimension 2.
\newblock {\em Comment. Math. Helv.}, 58:453–--502, 1983.

\bibitem{Jitormiskaya_Koslover_Schulteis_09}
S.~Jitomirskaya, D.~A. Koslover, and M.~S. Schulteis.
\newblock Continuity of the {L}yapunov exponent for analytic quasiperiodic
  cocycles.
\newblock {\em Ergodic Theory Dynam. Systems}, 29(6):1881--1905, 2009.

\bibitem{Johnson_1984}
R.~A. Johnson.
\newblock Lyapounov numbers for the almost periodic {S}chr\"odinger equation.
\newblock {\em Illinois J. Math.}, 28(3):397--419, 1984.

\bibitem{OhlsonTimoudas_17}
T.~Ohlson~Timoudas.
\newblock Power law asymptotics in the creation of strange attractors in the
  quasi-periodically forced quadratic family.
\newblock {\em Nonlinearity}, 30(12):4483, 2017.

\bibitem{OhlsonTimoudas_18}
T.~Ohlson~Timoudas.
\newblock Asymptotic laws for a class of quasi-periodic schrödinger cocycles
  at the lowest energy of the spectrum.
\newblock {\em preprint}, 2018.

\bibitem{Wang_You_13}
Y.~Wang and J.~You.
\newblock Examples of discontinuity of lyapunov exponent in smooth
  quasiperiodic cocycles.
\newblock {\em Duke Math. J.}, 162(13):2363--2412, 10 2013.

\end{thebibliography}
\bibliographystyle{abbrv}

\end{document}